\documentclass[a4paper,11pt]{article}

\usepackage{amssymb}
\usepackage{amsthm}
\usepackage {amsmath}
\usepackage{graphicx}

\DeclareMathOperator{\R}{\mathbb{R}}

\DeclareMathOperator{\A}{\mathcal{A}}

\DeclareMathOperator*{\argmin}{arg\,min}

\newcommand{\I}{{\mathcal I}}

\DeclareMathOperator*{\medcup}{\mathbin{\scalebox{0.85}{\ensuremath{\bigcup}}}}

\newcommand{\cI}{{\mathcal I}}
\newcommand{\cJ}{{\mathcal J}}
\newcommand{\cA}{{\mathcal A}}
\newcommand{\ccB}{{\mathcal B}}

\newcommand{\cM}{{\mathcal M}}

\newcommand{\cT}{{\mathcal T}}

\renewcommand{\phi}{\varphi}

\newtheorem{remark}{\textbf{Remark}}

\newtheorem{theorem}{\textbf{Theorem}}

\newtheorem{proposition}{\textbf{Proposition}}

\newtheorem{definition}{\textbf{Definition}}

\newtheorem{example}{\textit{Example}}

\begin{document}




\title{Domain decomposition based parallel Howard's algorithm}


\author{Adriano Festa, \\
\'Ecole Nationale Sup\'erieure de
Techniques Avanc\'ees\\ ENSTA ParisTech\\
              \emph{adriano.festa@ensta.fr}             \\
             828, Boulevard des Maréchaux,\\
91120 Palaiseau}
\maketitle
\begin{abstract}
The Classic Howard's algorithm, a technique of resolution for discrete Hamil\-ton-Jacobi
equations, is of large
use in applications for its high efficiency and good performances. A
special beneficial characteristic of the method is the superlinear convergence which, in presence of a finite number of controls, is reached in finite time.
Performances of the method can be significantly improved by using
parallel computing; how to build a parallel version of method is not a trivial point, the difficulties come from the strict relation between various values of the solution,
even related to distant points of the domain.
In this contribution we propose a  parallel version of the
Howard's algorithm driven by an idea of domain decomposition. This permits
to derive some important properties and to prove the
convergence under quite standard assumptions. The good features of the algorithm will be shown through some tests
and examples.
\end{abstract}

\par
\medskip

\section{Introduction}
The \emph{Howard's algorithm} (also called \emph{policy iteration algorithm}) is a classical method for solving a discrete Hamilton-Jacobi equation. This technique, developed by Bellman and Howard \cite{B57,H60}, is of large use in applications, thanks to its good proprieties of efficiency and simplicity.\par

It was clear from the beginning that in presence of a space of controls with infinite elements, the convergence of the algorithm is comparable to Newton's method. This was shown under progressively more general assumptions \cite{PB79,SR04} until to \cite{BMZ09}, where using the concept of \emph{slant differentiability} introduced in \cite{Q93, QS93}, the technique can be shown to be of semi-smooth Newton's type, with all the good qualities in term of superlinear convergence and, in some cases of interest, even quadratic convergence. \par
In this paper, we propose a parallel version of the policy iteration algorithm, discussing the advantages and the weak points of such proposal.\par
In order to build such parallel algorithm, we will use a theoretical construction inspired by some recent results on domain decomposition (for example \cite{BBC13,RaoZidani,BBC14}). Anyway, for our purposes, thanks to a greater regularity of the Hamiltonian, the decomposition can be studied just using standard techniques. We will focus instead on convergence of the numerical iteration, discussing some sufficient conditions, the number of iteration necessary, the speed.\par
Parallel Computing applied to Hamilton Jacobi equations is a subject of actual interest  because of the strict limitation of classical techniques in real problems, where the memory storage restrictions and limits in the CPU speed, cause easily the infeasibility of the computation, even in cases relatively easy.  With the purpose to build a parallel solver, the main problem  to deal with is to manage the information passing through the threads. Our analysis is not the first contribution on the topic, but it is an original study of the specific possibilities offered by the Policy algorithm. In particular some non trivial questions are: is convergence always guaranteed? In finite time? With which rate? Which is the gain respect to (the already efficient) Classical Howard's Algorithm?\par
In literature, at our knowledge, the first parallel algorithm proposed was by Sun in 1993 \cite{S93} on the numerical solution of the Bellman equation related to an exit time problem for a diffusion process (i.e. for second order elliptic problems);  an immediately successive work is \cite{CFLS94} by Camilli, Falcone, Lanucara and Seghini, here an operator of the semiLagrangian kind is proposed and studied on the interfaces of splitting. More recently, the issue was discussed also by Zhou and Zhan \cite{ZZ03} where, passing to a quasi variational inequality formulation equivalent, there was possible a domain decomposition.\par
Our intention is to show a different way to approach the topic. Decomposing the problem directly in its differential form, effectively, it is possible to give an easy and consistent interpretation to the condition to impose on the boundaries of the sub-domains. Thereafter, passing to a discrete version of such decomposed problem it becomes relatively easy to show the convergence of the technique to the correct solution, avoiding the technical problems, elsewhere observed, about the manner to exchange information between the sub-domains. In our technique, as explained later, we will substitute it with the resolution of an auxiliary problem living in the interface of connection in the domain decomposition. In this way, data will be passed implicitly through the sub-problems. \par
The paper is structured as follows: in section 2 we recall the classic Howard's algorithm and the relation with the differential problem, focusing on the case of its Control Theory interpretation. In section 3, after discussing briefly the strategy of decomposition, we present the algorithm, and we study the convergence. Section 4 is dedicated to a presentation of the performances and to show the advantages with respect the non parallel version. We will end presenting some possible extensions of the technique to some problems of interest: reachability problems with obstacle avoidance, max-min problems.

\section{Classic Howard's algorithm}\label{s:1}
The problem considered is the following. Let be $\Omega$ bounded open  domain  of $\R^d$ ($d\geq1$);  the steady, first order,  \emph{Hamilton-Jacobi equation} (HJ) is:
\begin{equation}\label{hj}
\left\{\begin{array}{ll}
   \lambda v(x)+H(x,Dv(x))=0 & x\in \Omega, \\
  v(x)=g(x) & x\in\partial \Omega,
\end{array}\right.
\end{equation}
where, following its \emph{Optimal Control interpretation}, $\lambda\in\R^+$ is the \emph{discount factor},  $g:\Omega \to \R$ is the \emph{exit cost}, and the \emph{Hamiltonian} $H:\Omega\times \R^d \to \R$ is defined by: $H(x,p):=\inf_{\alpha\in\cA}\{-f(x,\alpha)\cdot p-l(x,\alpha)\}$ 
with $f:\Omega\times\cA\rightarrow \R$ (\emph{dynamics})  and $l:\Omega\times\cA\rightarrow \R$ (\emph{running cost}). The choice of such Hamiltonian is not restrictive but useful to simplify the presentation. As extension of the techniques we are going to present, it will be shown, in the dedicate section, as the same results can be obtained in presence of different kind of Hamiltonians, as in obstacle problems or in differential games.\par
Under classical assumptions on the data (for our purposes we can suppose $f(\cdot,\cdot)$ and $l(\cdot,\cdot)$ continuous, $f(\cdot,\alpha)$ and $l(\cdot,\alpha)$ Lipschitz continuous for all $\alpha\in\cA$ and verified the \emph{Soner's condition} \cite{S86}), it is known (see also \cite{BCD97}, \cite{E98}) that the equation \eqref{hj} admits a unique continuous solution $v:\overline\Omega\rightarrow \R$ in the \emph{viscosity solutions} sense. \par
The solution $v$ is the value function to the infinite horizon problem with exit cost, where $\tau_x$ is the \emph{first time of exit} form $\Omega$: 
$$\left\{
\begin{array}{l}
\displaystyle v(x)=\inf\limits_{a(\cdot) \in L^{\infty}([0,+\infty[;\cA)}\int\limits_0^{\tau_x(a)}l(y_x(s),a(s))e^{-\lambda s}\,ds+e^{-\lambda \tau_x(a)}g(y_x(\tau_x(a))), \\
\hbox{where $y_x(\cdot)$ is a.e. solution of } \quad \left\{\begin{array} {l}\dot{y}(t)=f(y(t),a(t))\\  y(0)=x \end{array}\right. .
\end{array}
\right. 
$$

Numerical schemes for approximation of such problem have been proposed from the early steps of the theory, let us mention the classical Finite Differences Schemes \cite{CL84,S85},  semiLagrangian \cite{FF13}, Discontinuous Galerkin \cite{CS07} and many others. \par
In this paper we will focus on a \emph{monotone}, \emph{consistent} and \emph{stable} scheme (class including the first two mentioned above), which will provide us the discrete problem where to apply the Howard's Algorithm.\par
Considered a discrete grid $G$ with $N$ points $x_j$, $j=1,...,N$ on the domain $\overline{\Omega}$, the finite $N$-dimensional approximation of $v$,  $V$, will be the solution of the following discrete equation  ($V_j=V(x_j)$)
\begin{equation}\label{scheme}
F^h_i(V_1,...,V_N)=F^h_i(V)=0, \qquad i\in\{1,...,N\},
\end{equation}
where $h:=\max diam S_j$, (maximal diameter of the family of simplices $S_j$ built on $G$) is the discretization step, and related to a subset of the $V_j$, there are included the Dirichlet conditions following the obvious pattern
$$ F^h_j(V_1,...,V_N):=g(x_j), \quad {x_j\in \partial \Omega}.$$
We will assume on $F$, some Hypotheses sufficient to ensure the convergence of the discretization
\begin{itemize}
\item[{\bf (H1*)}] \emph{Monotony.} For every choice of two vectors $V,W$ such that, $V\geq W$ (component-wise) then $F^h_i(V_1,...,V_N)\geq F^h_i(W_1,...,W_N)$ for all $i\in \{1,...,N\}$.
\smallskip
\item[{\bf (H2*)}] \emph{Stability.} If the data of the problem are finite, for every vector $V$, there exists a $C\geq 0$ such that $V$, solution of \eqref{scheme}, is bounded by $C$ i.e. $\|V\|_{\infty}=\max_{i=1,...,N}|V_i|\leq C$ independently from $h$.
\smallskip
\item[{\bf (H3)}] \emph{Consistency.}  This hypothesis, not necessary in the analysis of the convergence of the scheme, is essential to guarantee that the numerical solution obtained approximates the continuous solution.  It is assumed that $F^h_i(\phi(y_1)+\xi, ..., \phi(y_N)+\xi)\rightarrow \lambda \phi(x_i)+H(x_i,\phi(x_i), D\phi(x_i))$ for every $\phi\in C^1(\Omega)$, $x_i\in\Omega$, with $h\rightarrow 0^+$, $y_i\rightarrow x_i$, and $\xi\rightarrow 0^+$.
\end{itemize}

Under these assumptions it has been discussed and proved \cite{S85} that  $V$, solution of \eqref{scheme}, converges to $v$, viscosity solution of \eqref{hj} for $h\rightarrow 0$. \par

The special form of the Hamiltonian $H$ gives us a correspondent special structure of the scheme $F$, in particular, with a rearrangement of the terms, the discrete problem \eqref{scheme} can be written as a resolution of a nonlinear system in the following form:
\begin{equation}\label{HOW}
\hbox{Find }V\in\R^N;\quad\min_{\alpha\in\A^N}(B(\alpha)V-c_g(\alpha))=0,
\end{equation}
where  $B$ is a $N\times N$ matrix and $c_g$ is a $N$ vector. The name $c_g$ is chosen to underline (it will be important in the following) that such vector there are contained information about the Dirichlet conditions imposed on the boundaries.
The \emph{Policy Iteration Algorithm} (or Howard's Algorithm)  consists in a two-steps iteration with an alternating improvement of the policy and the value function, as shown in Table \ref{HA}.
%
%

It is by now known \cite{BMZ09} that under a monotonicity assumption on the matrices $B(\alpha)$, (we recall that a matrix is monotone if ans only if it is invertible and every element of its inverse are non negative), automatically derived from (H1*) (as shown below), the above algorithm is a non smooth Newton method that converges superlinearly to the discrete solution of problem. The convergence of the algorithm is also discussed in the earlier work \cite{SR04, PB79} where the results are given in a more regular framework.\par Additionally, if $\cA$ has a finite number of elements, and this is the standard case of a discretized space of the controls,  then the algorithm converges in a finite number of iterations. \par

Let us state, for a fixed vector $V\in\R^n$ the subspace of controls $\cA(V):=\argmin B(\alpha)V-c_g(\alpha)$

\begin{proposition}\label{p:1}
Let us assume the matrix $B(\alpha)$ is invertible. If (H1*) holds true, then $B(\alpha)$ is monotone and not null for every $\alpha\in \cA(V)$ with  $V\in\R^n$.\par
\end{proposition}
\begin{proof}
For a positive vector $V$, consider a vector $W$ such that  $W-V\geq 0$ componentwise, then for H1*
$$ B(\bar \alpha) W-c_g(\bar \alpha)\geq \min_{\alpha\in\cA} B(\alpha) W-c_g(\alpha)\geq \min_{\alpha\in\cA} B(\alpha) V-c_g(\alpha)=B(\bar \alpha) V-c_g(\bar\alpha),$$
where $\bar \alpha\in \A(V)$, therefore
$$ B(\bar \alpha)(W-V)\geq 0. $$
Suppose now that the $i^{th}$ column of $B^{-1}(\bar \alpha)$ has a negative entry: choosing $W-V= e_i$ ($e_i$ $i^{th}$ column of the identity matrix) multiplying the previous relation for $B^{-1}(\bar \alpha)$ we have a contradiction. Then $B(\bar \alpha)$ is monotone.
\end{proof}

\begin{table}
\begin{center}
\line(1,0){380}\\
{\sc Howard's Algorithm (HA)}
\line(1,0){380}\\
\end{center}
Inputs: $B(\cdot)$, $c_g(\cdot)$. (Implicitly, the values of $V$ at the boundary points) \\
Initialize $V^0\in\R^N$ and $\alpha_0\in \cA^N$ \\
 Iterate $k\geq 0$:
\begin{itemize}
\item[i)]  Find $V^k\in\R^N$ solution of
$B(\alpha^k)V^k=c_g(\alpha^k)$.\\
If $k\geq 1$ and $V^k=V^{k-1}$, then stop. Otherwise go to (ii).
\item[ii)] 
$\alpha^{k+1}:=\argmin\limits_{\alpha\in\cA^n}
\left(B(\alpha)V^k-c_g(\alpha)\right)$.\\
Set $k:=k+1$ and go to (i)
\end{itemize}
Outputs: $V^{k+1}$.
\begin{center}
\line(1,0){380}\\
\end{center}
\caption{Pseudo-code of HA}\label{HA}
\end{table}

It is useful to underline the conceptual distinction between the convergence of the algorithm and the convergence of the numerical approximation to the continuous function $v$ as discussed previously. In general, the Howard's algorithm is an acceleration technique for the calculus of the approximate solution, the error with the analytic solution will be depending on the discretization scheme used. \par

To conclude this introductory section let us make two monodimensional basic examples.
\begin{example}[1D, Upwind scheme, Howard's Algorithm]\label{ex1} 
An example for the matrix $B(\alpha)$ and the vector $c_g(\alpha)$ is the easy
case of an upwind explicit Euler scheme in dimension one
\begin{equation}\label{DFf}\left\{
\begin{array}{l}
V_0=g(x_0)\\
   \lambda V_i=\min\limits_{\alpha_i\in
\cA}\left(l(x_i,\alpha_i)+f^+_i(\alpha_i)\frac{V_{i+1}-V_i}{h}+f^-_i(\alpha_i)\frac{V_i-V_{i-1}}{h}\right),\quad i\in\{2,...,N-1\}\\
V_{N}=g(x_{N})
\end{array}\right.
\end{equation}
where $x_i$ is a uniform discrete grid consisting in $N$ knots of distance $h$. Moreover, $f^+_i(\alpha_i)=\max\{0,f(x_i,\alpha_i)\}$ and
$f^-_i(\alpha_i)=\min\{0,f(x_i,\alpha_i)\}$. In this case the system \eqref{HOW} is
\begin{equation*}
B(\alpha)=\left(\begin{array}{ccccc}
1+\frac{\left[f^+_1-f^-_1\right]}{ h\lambda} & -\frac{f^+_1}{h\lambda} & 0 & \cdots & 0
\\
\frac{f^-_2}{h\lambda} & 1+\frac{\left[f^+_2-f^-_2\right]}{h\lambda} &
-\frac{f^+_2}{h\lambda}  & \cdots & 0\\
0 & \ddots & \ddots & \ddots & 0\\
0 & \cdots & \cdots & \frac{f^-_N}{h\lambda} &
1+\frac{\left[f^+_N-f^-_N\right]}{h\lambda}
\end{array} \right),
\end{equation*}
and 
\begin{equation*}\label{ceul}
c_g(\alpha)=\frac{1}{\lambda}\left(\begin{array}{c}
-{f^-_1}\;g(x_0)+{l(x_1,\alpha_1)}\\
{l(x_2,\alpha_2)}\\
\vdots\\
{l(x_{N-1},\alpha_{N-1})}\\
+{f^+_N}\;g(x_{N+1})+{l(x_N,\alpha_{N})}\end{array}
\right) .
\end{equation*}
It is straightforward that the solution of Howard's algorithm, verifying $\min_\alpha B(\alpha)V-c_g=0$, is the solution of \eqref{DFf}.
\end{example}

\begin{example}[1D, Semilagrangian, Howard's Algorithm]\label{ex2} 
If we consider the standard 1D semiLagrangian scheme, the matrix $B(\alpha)$ and the vector $c_g(\alpha)$ are
\begin{equation*}
B(\alpha)=\left(\begin{array}{ccccc}
1-\beta b_1(\alpha_1)& -\beta b_2(\alpha_1) & \cdots & -\beta b_N(\alpha_1)
\\
-\beta b_1(\alpha_2)  & 1-\beta b_2(\alpha_2) &
\cdots &-\beta b_N(\alpha_2) \\
\ddots & \ddots & \ddots & \ddots\\
-\beta b_1(\alpha_N)&  \cdots & -\beta b_{N-1}(\alpha_N) &
1-\beta b_N(\alpha_N) 
\end{array} \right),
\end{equation*}
and 
\begin{equation*}\label{ceul}
c_g(\alpha)=\left(\begin{array}{c}
h l(x_1,\alpha_1)+\beta b_0(\alpha_1)g(x_0)\\
h l(x_2,\alpha_2)\\
\vdots\\
h l(x_{N-1},\alpha_{N-1})\\
h l(x_N,\alpha_N)+\beta b_{N+1}(\alpha_N)g(x_{N+1})\end{array}
\right),
\end{equation*}
where $\beta:=(1-\lambda h)$ and the coefficients $b_i$ are the weights of a chosen interpolation $\mathbb I[V](x_i+h f(x_i,\alpha_j))=\sum_{i=0}^{N+1}b_i(\alpha_j)V_i$.
\end{example}

Despite the good performances of the Policy Algorithm as a \emph{speeding up} technique, in particular in presence of a convenient initialization (as shown for example in \cite{KAF13}) an awkward limit appears naturally:  the necessity to store data of very big size. \par
Just to give an idea of the dimensions of the data managed it is sufficient consider that for a 3D problem solved on a squared grid of side $n$, for example, it would be necessary to manage a $n^3\times n^3$ matrix, task which  becomes soon infeasible,  increasing $n$. This give us an evident motivation to investigate the possibility to solve the problem in parallel, containing the complexity of the sub problems and the memory storage.

\section{Domain Decomposition and Parallel version} \label{sect:2}

The strict relation between  various points of the domain  displayed by equation \eqref{hj}, makes the problem to find a parallel version of the technique, not an easy task to accomplish. The main problem, in particular, will be about passing information between the threads, necessary without a prior knowledge  of the characteristics of the problem. 

Our idea is to  combine the policy iteration algorithm with a domain decomposition principle for HJ equations. Using the theoretical framework of the resolution of Partial Differential Equations on submanifolds, presented for example in \cite{RaoZidani,BBC13}, we consider a decomposition of $\Omega$ on a collection of subdomains:  
\begin{equation}\label{dec}
\Omega:=\medcup_{i=1}^{M_\Omega}
\Omega_i \medcup_{j=1}^{M_\Gamma}\Gamma_j, \quad \mbox{with } \stackrel{\circ}{\Omega}_i\cap\stackrel{\circ}{\Omega}_j=\emptyset, \quad \hbox{ for }i\neq j. 
\end{equation}

Where the interfaces $\Gamma_j$, $j=1,\cdots,M_\Gamma$ are some strata of dimension lower than $d$ defined as the intersection of two subdomains $\overline{\Omega}_i\cap\overline{\Omega}_k$ for $i\neq k$. 

The notion of viscosity solution on the manifold, in this regular case, will be coherent with the definition elsewhere

\begin{definition}
A upper semicontinuos function  $u$ in $\Gamma$ is a subsolution on $\Gamma$ if  for any $\phi\in C^1(\R^{d})$, any $\delta>0$ sufficiently small and any maximum point $x_0\in \Gamma_\delta:=\{x \, s.t. \, |x-y|<\delta, y\in\Gamma\}$ of $x\rightarrow u(x)-\phi(x)$ is verified
$$\lambda \phi(x_0)+H^\delta(x_0,D\phi(x_0))\leq 0,$$
where with $H^\delta(\cdot,\cdot)$ we indicate the Hamiltonian $H$ restricted on $\Gamma_\delta$.\\
The definition of supersolution is made accordingly.
\end{definition}

\begin{remark}
It is useful to underline that, differently from multidomains problems (like the already quoted \cite{BBC13,RaoZidani}) there is no need to use a specific concept of solutions through the interfaces. Thanks to the regularity of the Hamiltonian, the simple definition of viscosity solution on an enlargement of $\Gamma$ (called $\Gamma_\delta$) will be effective; as described by the following result.
\end{remark}

\begin{theorem}\label{T:1}
Let us consider a domain decomposition as stated in \eqref{dec}. The continuous function $\overline{v}:\Omega\rightarrow \R$, verifying, for a $\delta>0$,  in the viscosity sense  the system below
\begin{equation}\label{hjd}
\left\{ \begin{array}{ll}
\lambda \overline{v}(x)+H(x,D\overline{v}(x))=0 &x\in\Omega_i, i=1,...,M_\Omega\\
\lambda \overline{v}(x)+H^\delta(x,D\overline{v}(x))=0 &x\in\Gamma_j, j=1,...,M_\Gamma,\\
\bar{v}(x)=g(x), & x\in\partial \Omega,
\end{array}\right.
\end{equation}
 is coincident with the viscosity solution $v(x)$ of \eqref{hj}.
\end{theorem}

\begin{proof}
It is necessary  to prove the uniqueness of a continuous viscosity solution for \eqref{hjd}. After that, just invoking the existence and uniqueness results for the solution $v$  (solution of the original problem), and observing that it is also a continuous viscosity solution of the system, from coincidence on the boundary, we get thesis.\par
To prove the uniqueness it is possible to use the classical argument of ``doubling of variables''. We recall the main steps of the technique for the convenience of the reader. For two continuous viscosity solutions $\bar{u},\bar{v}$ of \eqref{hjd} using the auxiliary function
$$ \Phi_\epsilon(x,y):=\bar{u}(x)-\bar{v}(y)-\frac{|x-y|^2}{2\epsilon},$$
which has a maximum point in $(x_\epsilon,y_\epsilon)$, it is easy to see that 
$$\max_{x\in\overline{\Omega}}(\bar{u}-\bar{v})(x)=\max_{x\in\overline{\Omega}}\Phi_\epsilon(x,x)\leq \max_{x,y\in\overline{\Omega}}\Phi_\epsilon(x,y)=\Phi_\epsilon(x_\epsilon,y_\epsilon);$$
now the limit 
$$ \liminf_{\epsilon \rightarrow 0^+} \Phi_\epsilon(x_\epsilon,y_\epsilon)\leq 0,$$
is proved as usual deriving $\Phi_\epsilon$ and using the properties of sub supersolution, (for example, \cite{BCD97} Theo. II.3.1) with the observation that no additional difficulty appears when a subsequence $(x_{\epsilon_n},y_{\epsilon_n})$ is definitely in $\Gamma$ because of the regularity of the Hamiltonian through the same interface; for the possibility to exchange the role between $\bar{u}$ and $\bar{v}$ (both super and subsolutions) we have uniqueness.
\end{proof}


In the following section we propose a parallel algorithm based on the numerical resolution of the decomposed system above.  This technique consists of a two steps iteration: 
\begin{itemize}
\item[(i)] Use Howard's algorithm to solve in parallel ($n$ threads) the nonlinear systems obtained after discretization of \eqref{hjd} on the subdomains $\Omega_i$ (in this step  the values of $V$ are fixed on the boundaries); 
\item[(ii)] Update the values of $V$ on the interfaces of connection $\medcup_j\Gamma_j$ by using Howard's algorithm on the nonlinear system obtained from the second equation of  \eqref{hjd} (in this case the \emph{interior points} of $\Omega_j$ are constant).
\end{itemize}
As it is shown later, this two-step iteration permits the transfer of information through the interfaces performed by the phase (ii). This procedure, anyway, is not priceless, the number of the steps necessary for its resolution will be shown to be higher than the classic algorithm; the advantage will be in the resolution of smaller problems and the possibility of a resolution in parallel. Moreover, the coupling between phase (i) and (ii) produces a succession of results convergent in finite time, in the case of a finite space of controls.\par  The good performances of the algorithm, benefits and weak points  will be discussed in details in Section \ref{s:test}.

\subsection{Parallel Howard's Algorithm}
To describe precisely the algorithm it is necessary to state the following. Let us consider as before a uniform
grid $G:=\{x_j: j\in\cI\}$, the indices set $\cI:=\{0,...,N\}$, and a vector of
all the controls on the knots $\alpha:=(\alpha_1,...,\alpha_N)^T\in\cA^N$. 

The domain $\Omega$ is decomposed as $\Omega:=\cup_{i=1}^n
\Omega_i\cup \Gamma$, where, coherently with above $\Gamma:=\cup_{j=1}^{M_\Gamma}\Gamma_j$; this decomposition induces an similar structure in the
indices set $\cI:=\cI_1\cup\cI_2\cup...\cup...\cI_n\cup\cJ$, where every point
$x_k$ of index in $\cI_i$ is an ``\emph{interior point}'', in the sense that
for every $x_j\in B_{h}(x_k)$ (ball centred in $x_k$ of radius $h$, defined as previously), $j\in \cI_i$, for
every $j\neq k$. The set $\cJ$ is the set of all the ``\emph{border
points}'', which means, for a $i\in\cJ$ we have that there exists at least two
points $x_j$, $x_k\in B_h(x_i)$ such that
$j\in\cI_{j}$ and $k\in\cI_k$ with $j\neq k$.\par 

We will build $n$ discrete subproblems on the subdomains $\Omega_i$ using as described before a monotone, stable and consistent scheme. In this case a discretization
of the Hamiltonian provides, for every subdomain $\Omega_i$, related to points
$x_j$, $j\in\cI_i$, a matrix $\hat B_i(\hat \alpha_i)$ and a vector
$\hat c_i(\hat \alpha_i,\{V_j\}_{j\in\cJ})$. We highlighted here, the dependance of $c_i$
from the border points which are, either, points where there are imposed the
Dirichlet conditions (data of the problem) or points on the interface $\Gamma$
which have to be estimed. \par

Assumed for simplicity that every $\cI_i$ has the same number of $k$ elements, called $\bar k:=card(\cJ)$, we have $k:=\frac{N-\bar k}{n}$, and $\hat B_i(\cdot)\in\cM_{k\times k}$,
$\hat c_i(\cdot,\cdot)\in\R^k$. \par

In resolution over $\Gamma$ we will have a matrix
$\hat B_{n+1}(\hat \alpha_{n+1})$ and a relative vector
$\hat c_{n+1}(\hat \alpha_{n+1},\{V_j\}_{j\in\cI\setminus\cJ})$, in the
spaces, respectively, $\cM_{\bar k\times \bar k}$ and $\R^{\bar k}$. (For
the 1D case, e.g., we can easily verify that $\bar k=n-1$). 
In this framework, the numerical problem after the discretization of equations \eqref{hjd} is the following:
\par 
Find $V:=(V_1,...,V_i,...,V_n,V_{n+1})\in\R^{N}$ with $V_i=\{V_j\in\R^k| \; j\in\I_i\}$ for $ i=1,...,n$  and $V_{n+1}=\{V_j\in\R^{\bar k}|j\in\cJ\}$, solution of the following system of nonlinear equations:
\begin{equation}\label{HD}
\left\{ \begin{array}{l}
\min\limits_{\hat \alpha_i\in\cA^k}\left(\hat B_i(\hat \alpha_i)V_i-\hat c_i(\hat \alpha_i,V_{n+1}
)\right)=0, \qquad i=1,...,n;\\
\min\limits_{\hat \alpha_{n+1}\in\cA^{\bar k}}\left(\hat B_{n+1}(\hat \alpha_{n+1})V_{n+1}-\hat c_{n+1}(\hat \alpha_{
n+1},\{V_{j}\}_{j\in\{1,...n\}})\right)=0.   
\end{array}\right.
\end{equation}
\par

The resolution of first and the second equation of \eqref{HD} will be called respectively \emph{parallel part} and \emph{iterative part} of the method. 
Solving the parallel and the iterative part will be performed alternatively, as a double step solver. The iteration of the algorithm will generate a sequence $V^s\in\R^N$ solution of the two steps system

\begin{equation}\label{HD2}
\left\{ \begin{array}{ll}
\min\limits_{\alpha\in\cA^N}\left(B_i(\alpha)V^{s+2}-c_i(\alpha,V^{s+1}
)\right)=0, & i=1,...,n,\\
\min\limits_{\alpha\in\cA^N}\left(B_{n+1}(\alpha)V^{s+1}-c_{n+1}(\alpha,V^{s})\right)=0, &  \\
V^0=V_0. & 
\end{array}\right. 
\end{equation}
Where $B_i(\cdot)$, $c_i(\cdot,\cdot)$ are the matrices and vectors in $\cM^{N\times N}$, and $\R^N$, containing $\hat B_i(\cdot)$, $\hat c_i(\cdot,\cdot)$ and such to return as solution the argument of $c_i(\alpha,\cdot)$ elsewhere. Evidently, $B_i(\cdot)$  $c_{i}(\cdot,\cdot)$ with $i\in\{1,...,n\}$ are: equal to $\hat B_i$ in the $\{ik,..,(i+1)k-1\}\times \{ik,..,(i+1)k-1\}$ blocks, and equal to the rows $\mathbb I_i$ of the identity matrix elsewhere , $c_i=\hat c_i$ in the $\{ik,..,(i+1)k-1\}$ elements of the vector and $c_i(\cdot,V)=V$ elsewhere, (we call these entries, in the following \emph{identical arguments}); the same, in the $\{nk+1,..,N\}\times \{nk+1,..,N\}$ block, $\{nk,..,N\}$ elements of the vector for $i=n+1$. \par It is clear that, despite this formal presentation, made to simplify the notation in the following, each equation of \eqref{HD2}, negletting the trivial relations, is a nonlinear system on the same dimension than \eqref{HD}. Clearely, a solution of \eqref{HD} is the  fixed point of \eqref{HD2}. 

\par
\begin{remark}
   The convergence of the discrete problem above to the solution of equation \eqref{hjd}, for a consistent, monotone and stable scheme was proved by Souganidis in \cite{S85}), other examples are \cite{CL84,FF13}. It is consequent then, the domain decomposition result stated before gives the theoretical justification to the transition. Different issue will be to show the convergence of the method; point discussed in the following.
\end{remark}

\begin{table}
\begin{center}
\line(1,0){380}\\
{\sc Parallel Howard's Algorithm (PHA)}
\line(1,0){380}\\
\end{center}
Inputs: $\hat B_i(\cdot)$, $\hat c_i\cdot,V^k_{n+1})$ for $i=1,...,n+1$\\
Initialize $V^0\in\R^N$  and $\alpha^0$. \\
Iterate $k\geq 0$:
\begin{enumerate}
\item[1)] \emph{(Parallel Step)} for each $i=1,...,n$ \\
Call (HA) with inputs $B(\cdot)=\hat B_i(\cdot)$ and $c_g(\cdot)=\hat c_i(\cdot,\cdot)$ \\
Get $V^k_i=\{V^k(x_j)|j\in\cI_i\}.$
\item[2)] \emph{(Sequential Step)} \\
Call (HA) with inputs $B(\cdot)=\hat B_{n+1}(\cdot)$ and $c_g(\cdot)=\hat c_{n+1}(\cdot,\{V^k_i\}_{i=\{1,...,n\}})$\\
Get $V^k_{n+1}=\{V^k(x_j)|j\in\cJ\}.$
\item[3)]Compose the solution $V^{k+1}=(V^k_1,...,V^k_n,V^k_{n+1})$ \\
If $\|V^{k+1}-V^k\|_\infty\leq \epsilon$ then \emph{exit}, otherwise go to (1).
\end{enumerate}
Outputs: $V^{k+1}$
\begin{center}
\line(1,0){380}\\
\end{center}
\caption{Pseudo-code of PHA}\label{PHA}
\end{table}

It is evident that such technique can be expressed as 
\begin{equation*}
\left\{
\begin{array}{lll}
F^{h,i}_j(V^{s+2},V^{s+1})=0& & j\in \cI_i, \hbox{ with }i=1,...,n\\
F^{h,n+1}_j(V^{s+1},V^{s})=0 & & j\in \cJ
\end{array}
\right.
\end{equation*}
where, coherently with above $F^{h,i}_j(V,W):=\left[\min\limits_{\alpha\in\cA^N}\left(B_i(\alpha)V-c_i(\alpha,W
)\right)\right]_j$ for $j\in\cI_i$.
\par

\begin{remark}
 The hypotheses (H1*-H2*) will be naturally adapted to the new framework as below:
\begin{itemize}
\item[{\bf (H1)}] \emph{Monotony.} For every choice of two vectors $V,W$ such that, $V\geq W$ (component-wise) then $F^{h,i}_j(V,\cdot)\geq F^{h,i}_j(W,\cdot)$ for all $j\in \{1,...,N\}$, and $i=1,...,n+1$.
\smallskip
\item[{\bf (H2)}] \emph{Stability.} If the data of the problem are finite, for every vector $V$, and every $W$ s.t. $\|W\|_\infty\leq +\infty$, there exists a $C\geq 0$ such that $V$, solution of $F^{h,i}_j(V,W)=0$ with $j\in\{1,...,N\}$ and $i\in\{1,...,n+1\}$, is bounded by $C$ independently from $h$.
\smallskip
\end{itemize}
This will be sufficient, thanks also to H3, to ensure convergence of $(V_1,...,V_{n+1})$ solution of $\eqref{HD}$ to $\bar v$ for $h\rightarrow 0^+$.
\end{remark}

From the assumptions  on the discretization scheme some specific properties of $B_i(\cdot)$ and $c_{i}(\cdot,\cdot)$ can be derived

\begin{proposition}\label{pp}
Let us assume $H1 - H2$. Let state also
\begin{itemize}
\item[{\bf (H4) }]if $W_1\geq W_2$ then  $c_i(\alpha,W_1)\geq c_i(\alpha,W_2)$, for all $i=1,...,n+1$, for all $\alpha\in\cA$.
\end{itemize}
Then it holds true the following.
\begin{enumerate}
\item If invertible, the matrices $ B_i(\alpha)$  are monotone, not null for every $i\in\{1,...,n+1\}$, and for every $\alpha\in \cA\,\cap\, \argmin B_i(\alpha)V-c_i(\alpha,V)$ with $V\in\R^N_+$.\par
\item  If $\|W\|_\infty < +\infty$, we have that for all $i\in\{1,...,n+1\}$ and for every $\alpha\in \cA$, there exists a $C>0$ such that
\begin{equation}\label{bound} \|c_i(\alpha,W)\|_{\infty}\leq C \|B_i(\alpha)\|_\infty. \end{equation}
the same relation holds for $c_i(\cdot,W)$.
\item Called $V^*$ the fixed point of \eqref{HD2}, if we have $V\leq V^*$ (resp. $V\geq V^*$), then there exists a $\alpha\in\cA$ such that, for all $ i=1,...,n+1$,
\begin{equation}\label{comp}
   B_i(\alpha) V-c_i(\alpha,V)\leq 0 \quad \hbox{(resp. } B_i(\alpha) V-c_i(\alpha,V)\geq 0 \hbox{)}.
\end{equation}
\end{enumerate}
\end{proposition}

\begin{proof}
To prove $1$ let us just observing that the monotony of $\hat B_i(\cdot)$ is sufficient end necessary for the monotony of $B_i(\cdot)$, (elsewhere $B_i(\cdot)$ is a diagonal block matrix with all the other blocks invertible), then the argument is the same of Proposition \ref{p:1}, starting from two vectors $W-V:=\left(\begin{array}{l}W_1\\ W_2\end{array}\right)-\left(\begin{array}{l}V_1\\ V_2\end{array}\right)\in\R^N_+$ with the only difference that we need assumption H4 to get 
$$\hat B_i(\bar \alpha)(W_1-V_1)\geq \hat c_i(\bar \alpha,W_2 )-\hat c_i(\bar \alpha,V_2 )\geq 0, \quad \forall i=1,...,n+1;$$ or equivalently
$$ B_i(\bar \alpha)(W-V)\geq  c_i(\bar \alpha,W )- c_i(\bar \alpha,V )\geq 0, \quad \forall i=1,...,n+1;$$
then the thesis.\par
To prove 2, it is sufficient to see $c_i(\alpha,W)=B(\alpha)_i V$, then for H2 the thesis.
The proof of 3 is a direct consequence of monotony assumption H1 with the definition of $V^*$ as 
$$   B_i(\alpha) V^*-c_i(\alpha,V^*) = 0, \quad \forall i=1,...,n+1. $$
\end{proof} 
Here we introduce a convergence result for the (PHA) algorithm.
\begin{theorem}\label{t:1}
Assume that the function $\alpha\in\cA^N\rightarrow B_i(\alpha)\in\cM^{N\times N}$, with $B_i(\alpha)$ invertible, and
$(\alpha,x)\in\cA^N\times\R^n\rightarrow c_i(\alpha,x)\in\R^N$  are continuous on the variable $\alpha, x$ for $ i=1,...,n+1$, $\cA$ is a compact set of $\R^d$, and $(H1, H2, H4)$ hold. 

Then there exists a unique $V^*$ in $\R^N$ solution of \eqref{HD}. Moreover, the
sequence $V^k$ generated by the (PHA) \eqref{HD2} has the following properties:
\begin{itemize}
\item[(i)]Every element of the sequence $V^s$ is bounded by a constant $C$, i.e. $\|V^s\|_\infty \leq C <+\infty$.
\item[(ii)]If $V^0\leq V^*$ then $V^s\leq V^{s+1}$ for all $k\geq 0$, vice versa, if $V^0\geq V^*$ then $V^s\geq V^{s+1}$.
\item[(iii)] $V^s\rightarrow V^*$ when $s$ tends to $+\infty$.
\end{itemize}
\end{theorem}

\begin{proof}
The existence of a solution comes directly from the monotonicity of the matrices ${B}(\alpha)$, the existence of an inverse and then the existence of a solution of every system of \eqref{HD}. Let us show that such solution is limited as limit of a sequence of vectors of bounded norm.  Observing that, 
\begin{equation*}
\|V^s\|_\infty=\max\left\{\|V_i^s\|_{\infty}\right\}_{i=1,...,n+1}
\end{equation*}
Without loss of generality we assume that $\|V^s\|_{\infty}\equiv
\|V^s_{i^*}\|_\infty$. Considering the problem $$\min_{\alpha\in\cA} B_{i^*}(\alpha)V^s-c(\alpha,V^{s-1})=0,$$ we have for H2 that if $V^{s-1}$ is bounded then $\|V^{s}\|_\infty\leq C$. Adding that $V^0$ is chosen bounded, the thesis follows for induction.

Let us to pass now to prove the uniqueness: taken $V,W\in\R^N$ two solutions of \eqref{HD2}, we define the vector $W^*$ equal to $V$ in the identical arguments of $c_i(\alpha,\cdot)$ and equal to $W$ elsewhere, for a $i\in\{1,...,n+1\}$. We have that, for a control $\beta$ (for Proposition \ref{pp}.3),
$$ B_i(\beta)V-c_i(\beta,V)\geq 0\geq  B_i(\beta)W^*-c_i( \beta,W^*)= B_i(\beta)W-c_i(\beta,V)$$
then $ B_i(\beta)(V-W)\geq 0$ and for monotonicity $V\geq W$. Exchanging the role of $V$ and $W$, and for the arbitrary choice of $i$ (in some arguments the relation above is trivial) we get the thesis.

(i) To prove that $V^k\in\R^N$ is an increasing sequence is sufficient to prove that taken $V_1, V_2\in\R^N$ solution of 
$$\min_{\alpha\in\cA} B_i(\alpha)V_2-c_i(\alpha,V_1)=0 $$
with (the opposite case is analogue) $V_1\leq V^*$, for a choice of $i\in\{1,...,n+1\}$ is such that $V_2\geq V_1$. Let us observe, for a choice of $ \beta\in\cA$ and using \eqref{comp} of  Prop. \ref{pp}
\begin{multline*} 0=\min_{\alpha\in\cA} B_i(\alpha)V_2-c_i(\alpha,V_1)\\ \leq  B_i(\beta)V_2-c_i(\beta,V_1)\leq B_i(\beta)V_2-\left(B_i(\beta)V_1-c_i(\beta,V_1)\right)-c(\beta,V_1)
\end{multline*}
then $B_i(\beta)(V_2-V_1)\geq 0$ then $V_2\geq V_1$.\\
We need also to prove that $V_2\leq V^*$: if it should not be true, then, with a similar argument than above
$$0\geq B_i(\beta) V_2-c_i(\beta,V_1)\geq B_i(\beta) V_2-(B_i(\beta)V_2-c_i(\beta,V_2))-c_i(\beta,V_1)$$
then for H4, $V_1\geq V_2$ which contradicts what stated previously.
\end{proof}

It is also possible to show that the method stops to the fixed point in a finite time. This is an excellent feature of the technique; unfortunately, the estimate which is possible to guarantee is largely for excess and, although important from the theoretical point of view, not so effective to show the good qualities of the method. The performances will checked in the through some tests in the Section \ref{s:test}.

\begin{proposition}
If $Card(A)<+\infty$ and convergence requests of Theorem \ref{t:1} are verified, then $(PHA)$ converges to the solution in less than $Card(A)^N$ iterative steps.
\end{proposition}
\begin{proof}
The proof is slightly similar to the classic Howard's case (cf. for example \cite{BMZ09}).

Let us consider the abstract formulation $P:x\rightarrow y$, where $P(x)$ is determined by $N_P$ parameter in $A$, and $Q:y\rightarrow x$, where $Q(y)$ is determined by $N_Q$ parameter in $A$. Then if we consider the iteration
\begin{equation}\label{pppp}
\begin{array}{l}
P(x^k)=y^k\\
Q(y^k)=x^{k+1}
\end{array}
\end{equation}
and we suppose (Theorem \ref{t:1}) $x^k\leq x^{k+1}$, $y^k\leq y^{k+1}$; than called $\alpha^k$ the $N_P+N_Q$ variables in $A$ associated to $(x^k,y^k)$ we know that there exist a $k$ and a $l$ where $k<l\leq Card(A)^{N_P+N_Q}$, such that $\alpha^k=\alpha^l$, and again $(x^k,y^k)=(x^l,y^l)$. Afterwards $(x^k,y^k)$ is a fixed point of \eqref{pppp}.\\
To restrict to our case is sufficient identify the process $P$ with the (parallel) resolution on the sub-domains and $Q$ with the iteration on the interfaces between the sub domains.
\end{proof}

\begin{remark}
It is worth to notice that the above estimation is worse than the Classical Howard's case. In fact, the classical algorithm find the solution in $Card(A)^N$, the $(PHA)$ will have the same number of iterative steps. This number has to be multiplied, called $M_1$ the maximum number of nodes in a sub-domain and $M_2$ the number of nodes belonging to the interface, for $Card(A)^{(M_1+M_2)}$ getting, at the end, a total number of simple steps equal to  $Card(A)^{(N+M_1+M_2)}$, much more than the classical case. In this analysis we do not consider anyway, the good point of the decomposition technique, the fact that any computational step is referred to a smaller and simpler problem, with the evident advantages in term of time elapsed in every thread and memory storage needed.
\end{remark}

\section{Performances, tuning parameters}\label{s:test}

\begin{figure}[t]
\begin{center}
\includegraphics[height=4.5cm]{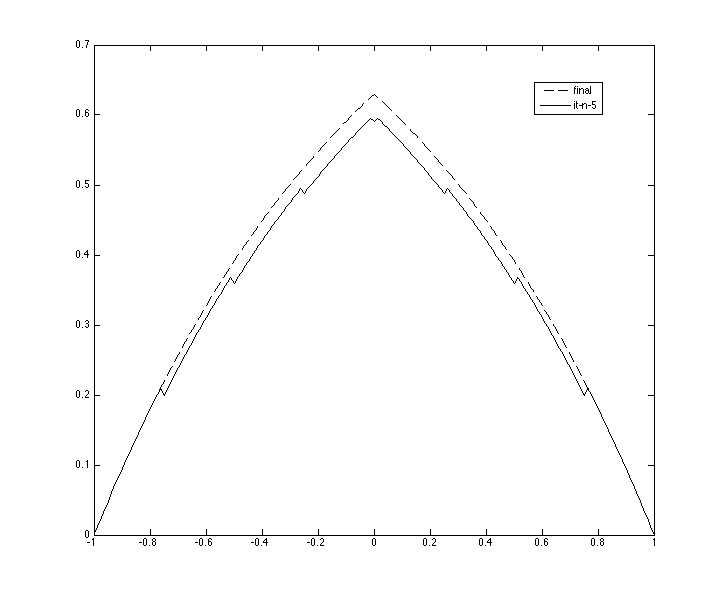}
\includegraphics[height=4.5cm]{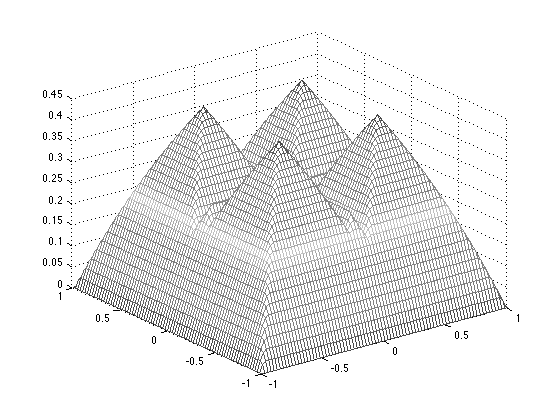}
\caption{Approximated solution of the iterative/parallel algorithm  (left) in the 1D case, final time (dotted) and fifth iteration (solid), in the 2D case (right, 3rd iteration).} \label{fig1}
\end{center}
\end{figure}

The performances of the algorithm and its characteristics as speeding up technique will be tested in this section. We will use a standard academic example where, anyway, there are present all the main characteristics of our technique. 
\paragraph{1D problem}
Consider the monodimensional problem 
\begin{equation}\left\{
\begin{array}{ll}
u(x)+|Du(x)|=1 & x\in (-1,1),\\
u(-1)=u(1)=0.
\end{array} \right.
\end{equation}
It is well known that this equation (\emph{Eikonal equation}) modelize the distance from the boundary of the domain, scaled by an exponential factor (\emph{Kruzkov transform}, cf. \cite{BCD97}).
Through a standard Euler discretization is obtained the problem in the form \eqref{HOW}.
In Table \ref{tt:2} is shown a comparison, in term of speed and efficacy, of our algorithm and the Classic Howard's one, in the case of a two thread resolution. It is possible appreciate as the parallel technique is not convenient in all the situations. This is due to the low number of parallel threads which are not sufficient to justify the construction. In the successive test, keeping fixed the parameter $dx$ and tuning number of threads it is possible to notice how much influential is such variable in terms of efficacy and time necessary for the resolution. \par

\begin{table}[ht]
\begin{center}\caption{Testing performances, 1D. Our method compared with the classic Howard's with two sub-problems. Efficacy compared in terms of time in seconds (t.), iterations (it.) relative to the parallel part of the algorithm (par.p.) and the iterative part (it.p.)}\label{tt:1}

\begin{tabular}{{c}|*{2}{c}|*4{c}}
        &   \multicolumn{2}{c|}{Classic HA}   &   \multicolumn{4}{c}{Parallel HA (2-threads)} \\
\hline
\hline
dx & time    & it. & t. (par. p.) &  it. (par. p.) & t. (it. p.) & Total t.\\
\hline
\bfseries 0.1  & e-3  &  10   &  1e-4& 4 &1e-5 & 1e-3 \\
\bfseries 0.05 & 6e-3 &  20 & 8e-4  &  5& e-5 & 3e-3 \\
\bfseries 0.025  &  0.09 & 40   &  7e-3 &  6 & 2e-5 & 0.04\\
\bfseries 0.0125  &  0.32   &   80  &   0.048 & 8 & 1e-4  &   0.36\\
\bfseries 0.00625 & 2.22 & 160 & 0.34 & 14 & 8e-4 & 3.26 \\
\end{tabular}
\end{center}
\end{table}

\begin{table}[h]
\begin{small}
\begin{center}\caption{Testing performances, 1D. Our method compared with the classic Howard's with various number of threads}\label{tt:2}
\vspace{0.2cm}
\begin{tabular}{{c}|*{2}{c}|*4{c}}

         dx=0.0125     &   \multicolumn{2}{c|}{Classic HA}   &   \multicolumn{4}{c}{Parallel HA} \\
\hline
\hline
threads & t.   & it. & t. (par. p.)  &  it. (par.) & t. (it. p.) & Total t.\\
\hline

\bfseries 2  &   &    &  0.48& 4 &1e-4 & 0.36 \\

\bfseries 4 &  &  & 8e-3  &  6& 1e-4 & 0.086 \\

\bfseries 8  &  0.32 &  80  &  18e-4 &  7 & 6e-4 & 0.014\\

\bfseries 16  &     &    &   7e-4 & 10 & 4e-4  &   \framebox{0.0095}\\

\bfseries 32 &  &  & 2e-4 & 8 & 6e-3 & 0.011 \\

\end{tabular}

\end{center}
\end{small}
\end{table}

In Table \ref{tt:2} we compare the iterations and the time (expressed in seconds as elsewhere in the paper) necessary to reach the approximated solution, analysing in the various phases of the algorithm, time and iterations necessary to solve every sub-problem (first two columns),  time elapsed for the iterative part (which passes the information through the threads, next column), finally the total time. It is highlighted the optimal choice of number of threads (16 thread); it is evident as that number will change with the change of the discretization step $dx$. Therefore it is useful to remark  that an additional work will be necessary to tune the number of threads accordingly to the peculiarities of the problem; otherwise the risk is to is to loose completely the gain obtained through parallel computing and to get worse performances even compared with the classical Howard's algorithm. \par 
As in the rest of the paper all the codes are developed in Mathworks' MATLAB\texttrademark  and performed on a processor 2,8 Ghz Intel Core i7; in the tests the parallelization is simulated.

\begin{table}[h]
\begin{center}\caption{Testing performances, 2D. Comparison with classical method and PH with 4 threads}\label{tt:3}
\begin{tabular}{{c}|*{2}{c}|*5{c}}

          &   \multicolumn{2}{c|}{Classic HA}   &   \multicolumn{5}{c}{Parallel HA (4-threads)} \\
\hline
\hline
dx & t.  & it. & t. (p.p.) &  it. (p.p.) & t. (it.p.) & it. (it.p.) & Total t.\\
\hline
\bfseries 0.1  & 0.05  &  11   &  0.009 & 8 &0.02 & 2 &0.04 \\
\bfseries 0.05 & 2.41 &  21 & 0.05  &  13& 0.03 & 2 & 0.14\\
\bfseries 0.025  &  73.3 & 40    &  2.5 &  22 & 0.15 & 3 & 7.83\\
\bfseries 0.0125 &   $>$e5 & - & 76 & 40 & 1.293 & 5 & 383.3 \\
\end{tabular}
\end{center}
\end{table}

\paragraph{2D problem}
The next test is in a space of higher dimension. Let us consider the approximation of the scaled distance function from the boundary of the square $\Omega:=(-1,1)\times(-1,1)$, solution of the eikonal equation
\begin{equation}\label{EIK}\left\{
\begin{array}{ll}
u(x)+\inf\limits_{a\in B(0,1)}\{-a\cdot Du(x)\}=1 & x\in \Omega,\\
u(x)=0 & x\in \partial \Omega.
\end{array} \right.
\end{equation}
where $B(0,1)\in\R^2$ is the usual unit ball. For the discretization of the problem is used a standard Euler discretization. Similar tests than the 1D case are performed, confirming the good features of our technique and, as already shown, the necessity of an appropriate number of threads with respect to the complexity of the resolution. \par

\begin{figure}[t]
\begin{center}
\includegraphics[height=4.5cm]{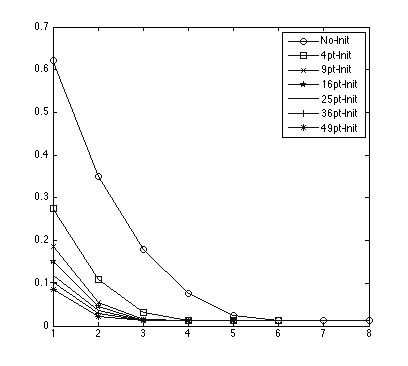}
\includegraphics[height=4.5cm]{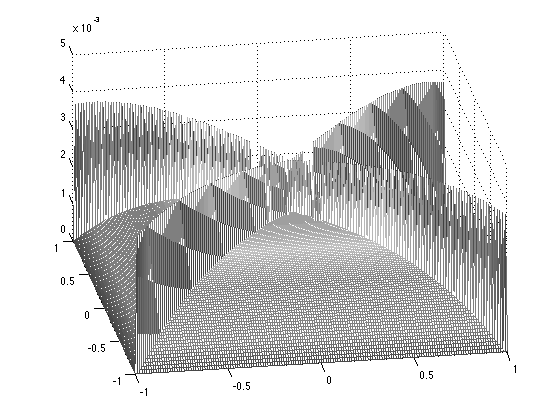}
\caption{Comparison with various initial guess to the speed of convergence of our method in the $L^2$-norm (left) and distribution of the error $dx=0.0125$, $16$ threads (right).  } \label{f:in}
\end{center}
\end{figure}

In Table \ref{tt:3} performances of the Classic Howard's algorithm are compared with our technique. In this case the number of threads are fixed to 4; the Parallel technique is evaluated in terms of: maximum time elapsed in one thread and max number of iterations necessary (first and second columns), time and number of iterations of the iterative part (third and fourth columns) and total time. In both the cases the control set $A:= B(0,1)$ is substituted by a $32-$points discrete version. It is evident, in the comparison, an improvement of the speed of the algorithm even larger than the simpler 1D case. This justifies, more than the 1D case, our proposal.

\begin{table}[th]
\begin{small}
\begin{center}\caption{Testing performances, 2D. Comparing different choices of the number of threads}\label{tt:4}
\vspace{0.2cm}
\begin{tabular}{{c}|*{2}{c}|*4{c}}

         dx=0.025     &   \multicolumn{2}{c|}{Classic HA.}   &   \multicolumn{4}{c}{Parallel HA} \\
\hline
\hline
threads & t.    & it. & t. (par. p.)  &  it. (par.) & t. (it. p.) & Total t.\\
\hline

\bfseries 4  &   &    &  2.5& 22 &0.15 & 7.83 \\

\bfseries 9 &  &  & 0.9 &  18& 0.5 & 5.08 \\

\bfseries 16  &  73.3 &  40  &0.05 &  13 & 1.6 & \framebox{1.826}\\

\bfseries 25  &     &    &   0.03 & 12 & 2.4  &   2.52\\

\bfseries 36 &  &  & 0.016 & 11 & 6.04 & 6.11 \\

\end{tabular}

\end{center}
\end{small}
\end{table}


In the Table \ref{tt:4} are compared  the performances for various choices of the number of threads, for a fixed $dx=0.025$. As in the 1D case is possible to see how an optimal choice of the number of threads can drastically strike down the time of convergence. In Figure \ref{f:in} is possible to see the distribution of the error. As is predictable, the highest concentration will correspond to the non-smooth points of the solution. It is possible to notice also how our technique apparently does not introduce any additional error in correspondence of the interfaces connecting the sub-domains. This is reasonable, although not evident theoretically. In fact, it is possible to prove the convergence of the scheme to the solution of \eqref{hjd} using classical techniques \cite{S85,FF13} but the rate of convergence could be different in the various subproblems, because of the (possibly different) local features of the problem.

\begin{remark}
   As shown in the tests, an important point of weakness of our technique is represented by the iterative part, which can be smaller and therefore easier than the ones solved in the parallel part, but it is highly influential in terms of general performances of the algorithm. In particular the number of the iterations of the coupling iterative-parallel part is sensible to a good initialization of the ``internal boundary'' points. As is shown in Figure \ref{f:in} a right initialization, even obtained on a very coarse grid, affects consistently  the overall performances. In this section, all the tests are made with a initialization of the solution on a $4^d$-points grid, with $d$ dimension of the domain space. The time necessary to compute the initial solution is always negligeable with respect to the global procedure.\par

\end{remark}

\begin{figure}[t]
\begin{center}
\includegraphics[height=4.5cm]{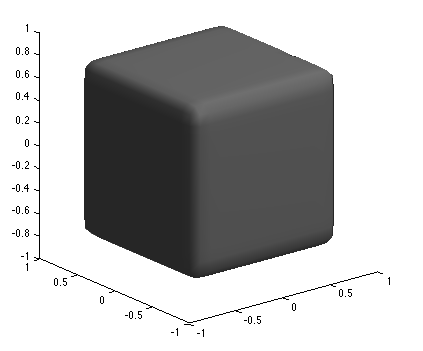}
\includegraphics[height=4.5cm]{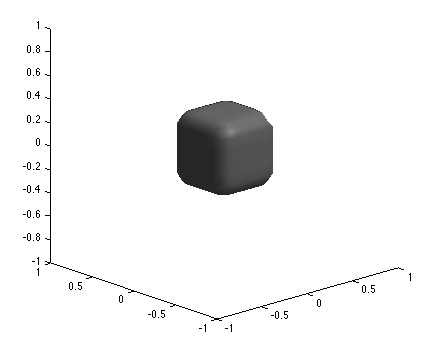}
\caption{Two level sets (corresponding to levels $u(x)=0.192$ (left) $u(x)=0.384$ (right)) of the approximated solution obtained with a $dx=0.1$ and an $8-$threads PHA. } \label{fig3d}
\end{center}
\end{figure}

\begin{table}[h]
\begin{center}\caption{Testing performances, 3D. Comparison with classical method and PI-H with 8 threads}\label{tt:5}
\vspace{0.2cm}
\begin{tabular}{{c}|*{2}{c}|*5{c}}

             &   \multicolumn{2}{c|}{Classic HA}   &   \multicolumn{5}{c}{Parallel HA (8-threads)} \\
\hline
\hline
dx & time  & it. & t. (p. p.) &  it. (p.p.) & t. (it. p.) & it. (it. p.) & Total t.\\
\hline
\bfseries 0.4  & 0.004  &  4   &  0.003 & 4 &0.002 & 1 &0.05 \\
\bfseries 0.2 & 0.22 &  6 & 0.026  &  6& 0.016 & 2 & 0.052\\
\bfseries 0.1  &  164.2 & 11   &  1.102 &  8 & 2.1 & 4 & 6.78\\
\bfseries 0.05 &   $>$e5 & - & 164 & 10 & 4.98 & 3 & 494 \\
\end{tabular}
\end{center}
\end{table}

\begin{figure}[th]
\begin{center}
\includegraphics[height=4.5cm]{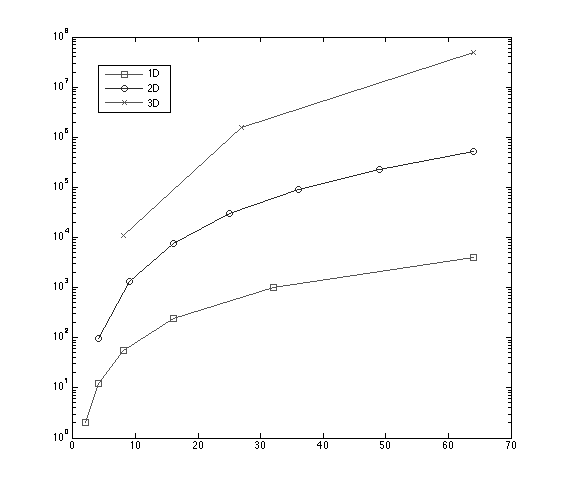}
\includegraphics[height=4.5cm]{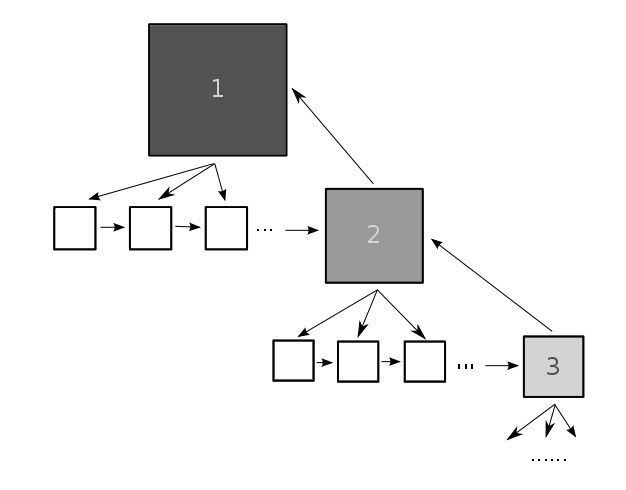}
\caption{Optimal number of splitting for number of variables in the discretization (left) and
iterative structure of the algorithm (right) to reduce the original problem (green) to a fixed number of variables sub-problems (blue).
} \label{fig4}
\end{center}
\end{figure}


\paragraph{3D problem} Analogue results are obtained also in the approximation of a 3D problem. Of course the effects of the increasing number of control points produces a greater complexity and will limit, for a same number of processors available, the possibility of a fine discretization of the domain.\\
Let us consider the domain $\Omega:=[-1,1]^3$ and the equation \eqref{EIK}, where $A:=B(0,1)$, unitary ball in $R^3$. In Figure \ref{fig3d} there are shown two level sets of the  solution obtained. A comparison with the performances of the Classic Howard's algorithm are shown in Table \ref{tt:4}.

\begin{remark}\label{multiloop}
   With the growth of the dimensionality of the problem a special care should be dedicated to the resolution of the iterative step. Suppose to simplify the procedure considering a square domain (in dimension $d=1,2,3,..$ an interval, a square, a cube..) and a successive splitting in equal regular subdomains. Calling $N$ the number of total variables and $N_s$ the number of the splitting (which generates a division in $N_s^d$ subdomains) the number of the elements in every thread of the parallel part is $\frac{N}{(N_s)^d}$, and the number of the variables in the iterative part $\frac{N}{\sqrt[d]{N}}(N_s-1)d$. Clearly the optimal choice of the number of threads  is such that the elements of the iterative part are balanced with the nodes in each subdomain, so it is straight forward to find the following optimal relation between number of splitting and total elements 
$$N=\left(N_s^d(N_s-1)d\right)^d.$$
It is evident that for a very high number of elements, (Figure \ref{fig4}), it is useless to use a great and non optimal number of threads. This contradiction comes from the bottleneck effect of the resolution on the interfaces of communication between the subdomains, indeed the complexity of such subproblem will grow with the number of threads instead to decrease, reducing our possibilities of resolution. The problem can be overcome with an additional parallel decomposition of the iterative pass, permitting us to decompose each subproblem to a complexity acceptable. Imagine to be able to solve (for computational reasons, memory storage, etc.) only problem of dimension ``white square'' (we refer to Figure \ref{fig4}, right) and to want to solve a bigger problem (``square 1'') with an arbitrary number of processors available. Through our technique we will decompose the problem in a finite number of subproblems ``white square'' and a (possibly bigger than the others) problem ``square 2''. We will replicate our parallel procedure for the ``square 2'' obtaining a collection of manageable problems and a ``square 3''. Through a reiteration of this idea we arrive to a decomposition in subproblems of dimension desired.
\end{remark}

\section{Extensions and Special Cases}
In this section there are shown some non trivial extensions to more general situations of the method. We will discuss, in particular, how to adapt the parallelization procedure to the case of a target problem, an obstacle problem and max-min problems, where the special structure of the Hamiltonian requires some cautions and remarks.

\subsection{Target problems}
An important class of problems where is useful to extend the techniques discussed is the Target problems where a trajectory is driven to arrive in a \emph{Target set} $\cT\subset\Omega$ optimizing a cost functional. \par
A easy way to modify our Algorithm to this case is to change the construction procedure for $B$ and $C$:
\begin{equation}\label{ch}
\left[B'(\alpha)\right]_i:=\left\{\begin{array}{ll}
\left[B(\alpha)\right]_i, & \hbox{ if }x_i\notin \cT, \\
\left[\mathbb{I}\right]_i, & \hbox{ otherwise;}\end{array}
\right. \;
c'(\alpha)_i:=\left\{\begin{array}{ll}
c(\alpha)_i, & \hbox{ if }x_i\notin \cT, \\
0, & \hbox{ otherwise;}\end{array}
\right. 
\end{equation}
this, with the classical further construction of \emph{ghost nodes} outside the domain $\Omega$ to avoid the exit of the trajectories from $\Omega$, will solve this case.\par
\begin{remark}
A question arises naturally in this modification: are the convergence results still valid? The answer is not completely  trivial because, for example, a monotone matrix modified as above is not automatically monotone (the easiest counterexample is the identical matrix flipped vertically: it is monotone because invertible and equal to its inverse, but changing any row as in \eqref{ch} we get a non invertible matrix). To prove the convergence it is sufficient to start from the numerical scheme associated to such modified algorithm. It is quite direct to show verified the hypotheses (H1-H4) getting as consequence the described properties of the algorithm.
\end{remark}

\begin{example}[Zermelo's Navigation Problem]
A well known benchmark in the field is the so-called Zermelo's navigation problem, the main feature, in this case, is that the dynamic is driven by a force of comparable power with respect to our control. The target to reach will be a ball of radius equal to $0.005$ centred in the origin, the control is in $ A=B(0,1) $. The other data are:
\begin{equation}
f(x,a)=a+\left(\begin{array}{c} 1-x_2^2 \\0 \end{array} \right), \quad \Omega=[-1,1]^2,\quad \lambda =1, \quad l(x,y,a)=1.
\end{equation}

\begin{figure}[t]
\begin{center}
\includegraphics[height=4.5cm]{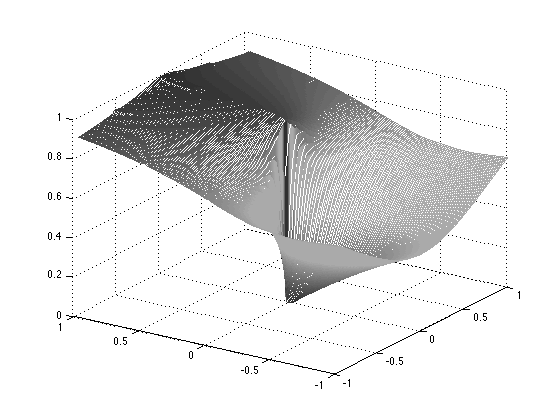}
\includegraphics[height=4.5cm]{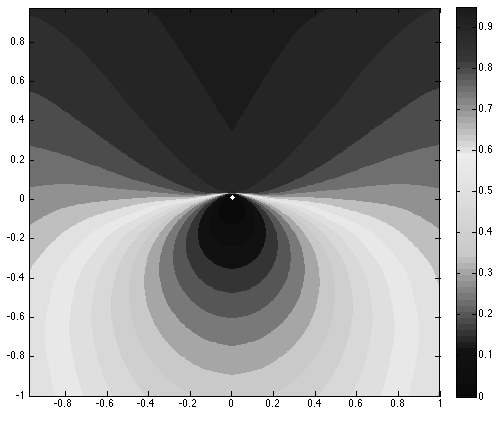}
\caption{Approximated solution for the Zermelo's navigation problem $dx=0,01$.} \label{fig1}
\end{center}
\end{figure}

In Table \ref{tt:6} a comparison with the number of threads chosen is made. Now we are in presence of characteristics not aligned with the grid, but the performances of the method are poorly effected. Convergence is archived with performances comparable to the already described case of the Eikonal Equation.

\begin{table}[th]
\begin{small}
\begin{center}\caption{Zermelo's navigation problem. Comparison of various choices of the number of threads}\label{tt:6}
\vspace{0.2cm}
\begin{tabular}{{c}|*{2}{c}|*4{c}}

         dx=0.025     &   \multicolumn{2}{c|}{Classic HA}   &   \multicolumn{4}{c}{Parallel HA} \\
\hline
\hline
threads & t.   & it. & t. (par. p.) &  it. (par.) & t. (it. p.) & Total t.\\
\hline

\bfseries 4  &   &    &  1.31& 11 &0.13 & 5.4 \\

\bfseries 9 &  &  & 0.7 &  9& 0.7& 4.2 \\

\bfseries 16  &  37.9 &  20  &0.031 &  7 & 1.38 & \framebox{1.53}\\

\bfseries 25  &     &    &   0.02 & 7 & 2.7  &   3.9\\

\bfseries 36 &  &  & 0.01 & 8 & 5.19 & 5.28 \\

\end{tabular}

\end{center}
\end{small}
\end{table}
\end{example}

\subsection{Obstacle Problem}
Dealing with an optimal problem with constraints using the Bellman's approach, various techniques have been proposed. In this section we will consider an implicit representation of the constraints through a level-set function.
Let us to consider the general single obstacle problem 
\begin{equation}\label{hj}
\left\{\begin{array}{ll}
  \max\left( \lambda v(x)+H(x,Dv(x)),v(x)-w(x)\right)=0 & x\in \Omega,\\
  v(x)=g(x) & x\in\partial \Omega,
\end{array}\right.
\end{equation}
where the Hamiltonian $H$ is of the form discussed in Section \ref{s:1} and the standard hypothesis about regularity of the terms involved are verified. The distinctive trait of this formulation is about the term $w(x):\Omega\rightarrow \R$, assumed regular, typically stated as the opposite of the signed distance from the boudary of a subset $K\subset \Omega$. The solution of this problem is coincident, where defined, with the solution of the same problem in the space $\Omega\setminus K$, explaining the name of ``obstacle problem'' (cf. \cite{CLY09}).\par
Through an approximation of the problem in a finite dimensional one, in a similar way as already explained, is found the following variation of the Howard's problem
\begin{equation}\label{HOW2}
\hbox{Find }V\in\R^N;\quad\min_{\alpha\in\A^N}\min(B(\alpha)V-c_g(\alpha),V-W)=0,
\end{equation}
where the term $W$ is a sampling of the function $w$ on the knot of the discretization grid. \par
It is direct to show that changing the definition of the matrix $B$ and $c$, is possible to come back to the problem \eqref{HOW}. Adding an auxiliary control to the set $\cA':=\cA\times \{0,1\}$ and re-defying the matrices $B$ and $c$ as 
\begin{equation}\label{ch1}
\begin{split}
\left[B'(\alpha)\right]_i:=\left\{\begin{array}{ll}
\left[B(\alpha)\right]_i, & \hbox{ if }B(\alpha)V-c_g(\alpha)\geq V-W \\
\left[\mathbb{I}\right]_i, & \hbox{ otherwise;}\end{array}
\right. 
\\
c_g'(\alpha)_i:=\left\{\begin{array}{ll}
c_g(\alpha)_i, & \hbox{ if }B(\alpha)V-c_g(\alpha)\geq V-W \\
W_i, & \hbox{ otherwise;}\end{array}
\right. 
\end{split}
\end{equation}
$$ \hbox{ for }i=1,...,N, $$
(where the $X_i$ is the $i-$row if $X$ is a matrix, and the $i-$ element if $X$ is a vector, and $\mathbb{I}$ is the identity matrix), the problem becomes 
\begin{equation}\label{HOWob}
\hbox{Find }V\in\R^N;\quad\min_{\alpha\in\cA'}(B'(\alpha)V-c_g'(\alpha))=0,
\end{equation}
which is in the form \eqref{HOW}. 
\begin{remark}
Even in this case the verification of Hypotheses (H1-H4) by the numerical scheme associated to the transformation \eqref{ch1} is sufficiently easy. It is in some cases also possible the direct verification of conditions of convergence in the obstacle problem deriving them from the free of constraints case. For example if we have that the matrix $B(\alpha)$ is strictly dominant (i.e. $A_{ij}\leq 0$ for every $j\neq i$, and there exists a $\delta>0$ such that for every $i$, $A_{ii}\geq \delta +\sum_{i\neq j}|A_{ij}|$), then the properties of the terms are automatically verified, (i.e. since all $B_i(\alpha)$ are strictly dominant and thus monotone).  

\end{remark}

\begin{example}[Dubin Car with obstacles]

A classical problem of interest is the optimization of trajectories modelled by 
\begin{equation*}
  f(x,y,z,a) :=
\left(\begin{array}{c}
c \cos(\pi z)\\
c \sin(\pi z)\\
a \end{array} \right), \quad \lambda:= 10^{-6}, \quad l(x,y,z,a):=1;
\end{equation*}
which produces a collection of curves in the plane $(x,y)$ with a constraint in the curvature of the path. Typically this is a simplified model of a car of constant velocity $c$ with a control in the steering wheel.\\
\begin{figure}[t]
\begin{center}
\includegraphics[height=4.5cm]{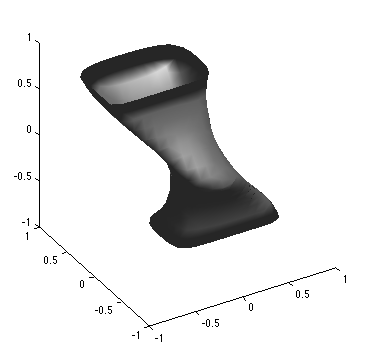}\hspace{0.5cm}
\includegraphics[height=4.5cm]{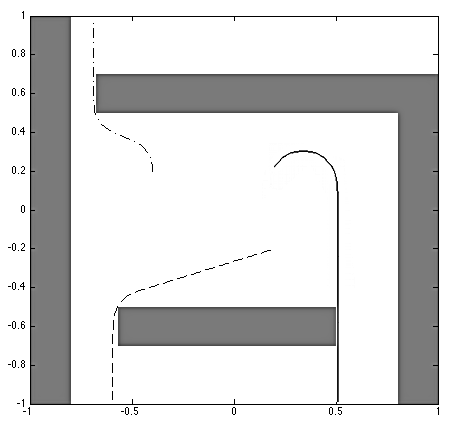}
\caption{Value function of Dubin Car Problem (left, free of constraints) and some optimal trajectories in the case with constraints (right).} \label{f:dub}
\end{center}
\end{figure}
The value function of the exit problem from the domain $\Omega:=(-1,1)^2$, $\mathcal A=[-1,1]$ discretized uniformly in 8 points is presented in Figure \ref{f:dub}. It is natural to imagine the same problem with the presence of constraints. Such problem can be handled with the technique described above producing the results shown in the same Figure \ref{f:dub}, where there are presented some optimal trajectories (in the space $(x,y)$) for the exit from $\Omega:=(-1,1)^2$ in presence of some constraints. From the picture it is possible to notice also the constraint about the minimal radius  of curvature contained in the dynamics.
\end{example}
\subsection{Max-min Problems}
The last, more complicated extension of the Howard's problem \eqref{HOW} is about max-min problems of the form 
\begin{table}[h]
\begin{center}
\line(1,0){380}\\
{\sc PHA (MaxMin case)}
\line(1,0){380}
\end{center}
Initialize $V^0\in\R^N$ $\alpha^0$ for
all $i\in\{1,...,n+1\}$.\\
k:=1;
\begin{enumerate}
\item[1)] Iterate \emph{(Parallel Step)} for every $i=1,...,n$ do: \\
$s:= 0$
\begin{itemize}
\item[1.i)]  Find $V^s_i\in\R^n$ solution of
$F_i^\beta(V^s_i)=0$.\\
If $s\geq 1$ and $V^s_i=V^{s-1}_i$, then $V_i:=V_i^{s}$, and exit (from inner loop). \\
Otherwise go to (1.ii).
\item[1.ii)] 
$\beta_i^{s+1}:=\argmin\limits_{\alpha\in\cA^n}
F_i^\beta(V^s_i)=0$.\\
Set $s:=s+1$ and go to (1.i)
\end{itemize}
\item[2)] Iterate \emph{(Sequential Step)} for $t\geq 0$
\begin{itemize}
\item[2i)] Find $V_{n+1}^t\in\R^h$ solution of
$F_{n+1}^\beta(V^t_{n+1})=0$.\\
If $t\geq 1$ and $V_{n+1}^t=V_{n+1}^{t-1}$, then $V_{n+1}=V^t_{n+1}$, and go to (3).\\ Otherwise go to (2ii).
\item[2ii)]
$\beta_{n+1}^{t+1}:=\argmin\limits_{\beta_{n+1}\in\ccB^h}F_{n+1}^\beta(V_{n+1})=0$.\\
Set $t:=t+1$ and go to (2i)
\end{itemize}
\item[3)]Compose the solution $V^{k+1}=(V_1, V_2, ..., V_n,V_{n+1})$\\
k:=k+1;\\
If $V^{k+1}=V^k$ then \emph{exit}, otherwise go to (1).
\end{enumerate}
\begin{center}
\line(1,0){380}
\end{center}
\caption{Pseudo-code of PHA for MaxMin problems.}\label{Mm}
\end{table}

\begin{equation}\label{HOWMM}
\hbox{Find } V\in\R^N; \quad  \max_{\beta\in\ccB^N}\left(\min_{\alpha\in\cA^N}\left(B(\alpha,\beta) V-c(\alpha,\beta)\right)\right)=0.
\end{equation}
Such a non linear equations arises in various contexts, for example in differential games and in robust control. The convergence of a Parallel algorithm for the resolution of such problem is also discussed in \cite{FS05}.\par
Also in this case, a modified version of the policy iteration algorithm can be shown to be convergent (cf. \cite{BMZ09}). Our aim in this subsection is to give some hints to build a parallel version of such procedure. \par

Let us introduce the function $F_i^{\beta}:\R^n\rightarrow \R$, for $\beta\in\ccB^n$ and $i\in\cI$ defined by
\begin{equation}\label{ff}
F^\beta_i(V):=\min_{\alpha\in\cA^n}(B_i(\alpha,\beta)V-c_i(\alpha,\beta,V)
\end{equation}
The problem \eqref{HOWMM}, in analogy with the previous case, is equivalent to solve the following system of nonlinear equations

\begin{equation}\label{HD3}
\left\{ \begin{array}{ll}
\min\limits_{\beta\in\ccB^k}F_i^\beta(V_i)=0 & i=1,...,n\\
\min\limits_{\beta\in\ccB^h}F_{n+1}^\beta(V_{n+1})=0 & 
\end{array}\right. 
\end{equation}
 
The Parallel Version of the Howard Algorithm in the case of a maxmin problem is summarized in Table \ref{Mm}.

\begin{remark} 
It is worth to notice that at every call of the function $F^\beta$ is necessary to solve a minimization problem over the set $\cA$, this can be performed in an approximated way, using, for instance, the classical Howard's algorithm. This gives to the dimension of this set a big relevance on the performances of our technique. For this reason, if the cardinality of $\cA$ (in the case of finite sets) is bigger than $\ccB$, it is worth to pass to the alternative problem $-max_{\alpha\in\cA}\min_{\beta\in\ccB}(B(\alpha,\beta)V-c(\alpha,\beta))$ (here there are used the Isaacs' conditions) before the resolution, inverting in this way, the role of $\cA$ and $\ccB$ in the resolution.
\end{remark}


\begin{example}[A Pursuit-Evasion game]
One of the most known example of max-min problem is the Pursuit evasion game; where two agents have the opposite goal to reduce/postpone the time of capture. The simplest situation is related to a dynamic
\begin{equation*}
  f(x,y,z,a,b) :=
\left(\begin{array}{c}
a_1/2-b_1\\
a_2/2-b_2 \end{array} \right)
\end{equation*}
where controls are taken in the unit ball $\mathcal A=\mathcal B=B(0,1)$ and capture happens when the trajectory is driven to touch the small ball $B(0,\rho)$, ($\rho = 0.15$, in this case). The passage to a Target problem is managed as described previously.
\begin{figure}[t]
\begin{center}
\includegraphics[height=4.5cm]{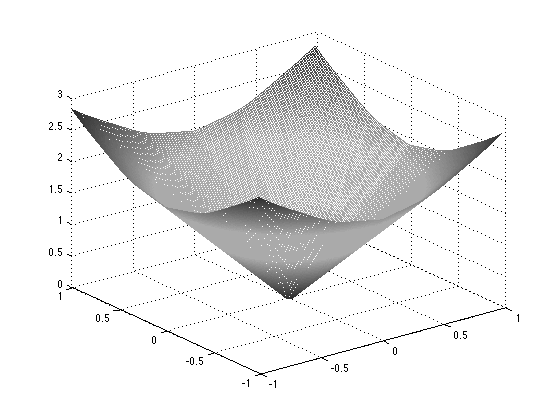}
\includegraphics[height=4.5cm]{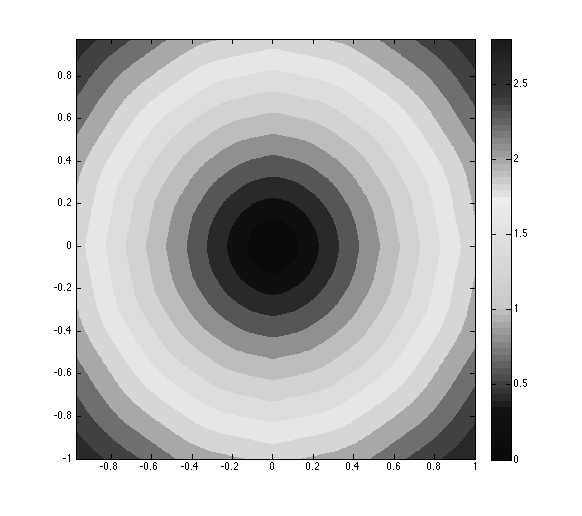}
\vspace{-0.5cm}
\caption{Approximated solution of the Pursuit Evasion game, $dx=0.0125$} \label{f:pe}
\end{center}
\end{figure}
In Figure \ref{f:pe} the approximated value function of that problem is shown.
\end{example}

\section{Conclusions}
The main difficulty in the use of the Howard's Algorithm, i.e. the resolution of big linear systems can be overcome using parallel computing. This is important despite the fact that we must accept an important drawback: the double loop procedure (or multi-loop procedure as sketched in remark \ref{multiloop}) does not permit to archive a superlinear convergence, as in the classical case; we suspect (as in Figure \ref{f:in}) that such rate is preserved looking to the (external) iterative step, where we have to consider, anyway, that in every step of the algorithm a resolution of a reduced problem is needed.\par
Another point influential in the technique is the manner chosen to solve every linear problem which appears in the algorithm. In this paper, being not in our intentions to show a comparison with other competitor methods rather studying the properties of the algorithm in relation of the classical case, we preferred the simplicity, using a routine based on the exact inversion of the matrix. Using of an iterative solver, with  the due caution about the error introduced, better performances are expected (cf. \cite{KAF13}). \par
Through the paper we showed as some basic properties of the schemes used to discretized the problem bear to sufficient conditions for the convergence of the algorithm proposed, this choice was made to try to keep our analysis as general as possible. A special treatment about the possibility of a domain decomposition in presence of non monotone schemes is possible, although not investigated here.

\section*{Acknowledgements}
This work was supported by the European Union under the 7th Framework Programme FP7-PEOPLE-2010-ITN SADCO, Sensitivity Analysis for Deterministic Controller Design.\\
The author thanks Hasnaa Zidani of the UMA Laboratory of ENSTA for the discussions and the support in developing the subject.

\end{document}